\documentclass[reqno, 12pt,a4letter]{amsart}
\usepackage{amsmath,amsxtra,amssymb,latexsym, amscd,amsthm}
\usepackage[utf8]{inputenc}
\usepackage[mathscr]{euscript}
\usepackage{mathrsfs}
\usepackage[english]{babel}
\usepackage{color}
\usepackage[active]{srcltx}
\usepackage{enumerate}
\usepackage{mathabx}
\usepackage{tikz}
\usetikzlibrary{calc}
\usetikzlibrary{math}

\newtheorem {theorem}{Theorem}[section]
\newtheorem {corollary}{Corollary}[section]
\newtheorem {lemma}{Lemma}[section]
\newtheorem {claim}{Claim}
\newtheorem {proposition}{Proposition}[section]

\theoremstyle{definition}
\newtheorem{definition}{Definition}[section]
\newtheorem{example}{Example}[section]
\newtheorem{remark}{Remark}[section]

\newcommand{\vol}{{\rm vol}}

\newcommand{\dist}{{\rm dist}}

\newcommand{\sing}{{\rm sing}}

\newcommand{\R}{\mathbb R}
\newcommand{\B}{\mathbb B}
\newcommand{\C}{\mathbb C}
\newcommand{\K}{\mathbb K}
\newcommand{\SA}{\mathcal S}

\def\EES{{\accent"5E e}\kern-.5em\raise.8ex\hbox{\char'23 }}
\def\ow{o\kern-.42em\raise.82ex\hbox{
   \vrule width .12em height .0ex depth .075ex \kern-0.16em \char'56}\kern-.07em}
\def\OW{o\kern-.460em\raise1.36ex\hbox{
\vrule width .13em height .0ex depth .075ex \kern-0.16em
\char'56}\kern-.07em}
\def\DD{D\kern-.7em\raise0.4ex\hbox{\char '55}\kern.33em}

\title{Some variational properties of tangent directions at infinity of real algebraic sets}

\author{S\~i Ti\d{\^e}p \DD inh$^\dagger$}
\address{Institute of Mathematics, VAST, 18, Hoang Quoc Viet Road, Cau Giay District 10307, Hanoi, Vietnam}
\address{Institute of Mathematics, Polish Academy of Sciences, Śniadeckich 8, 00-656 Warsaw, Poland}
\email{dstiep@math.ac.vn}

\author{Ti\EES n-S\OW n Ph\d{A}m$^{\ddag}$}
\address{Department of Mathematics, University of Dalat, 1 Phu Dong Thien Vuong, Dalat, Vietnam}
\email{sonpt@dlu.edu.vn}



\date{ \today}

\begin{document}

\begin{abstract} In this paper, we relate the set of asymptotic critical values of a polynomial function $f$ with the set of discontinuity of two functions, the multivalued function which associate to each value $t$ the set of tangent directions at infinity of the fiber $f^{-1}(t)$ and the composition of the $(n-2)$-dimensional volume function with the first one. This gives necessary conditions of equisingularity at infinity for the family of the fibers of a real polynomial function.
\end{abstract}

\maketitle

\pagestyle{plain}

\section{Introduction} 
Let $f:\K^n\to\K$ be a polynomial function where $\K=\C$ or $\K=\R$. It is well-known that $f$ is a local $C^\infty$-trivial fibration outside a finite subset of $\K$~\cite{Thom1969}, the smallest such set is called the bifurcation set of $f$, denoted by $B(f).$ In general the set $B(f)$ is larger than the set $K_0(f)$ of critical values of $f$ since it contains also the set $B_\infty(f)$ of bifurcation values at infinity of $f$. Roughly speaking, the set $B_\infty(f)$ consists of points at which $f$ is not a locally trivial bundle at infinity (i.e., outside a large ball). For $n=2$, the set $B_\infty(f)$ can be effectively computed in the complex case~\cite{Chadzynski2003,HaHV1984} as well as it has been described explicitly in the real case~\cite{Coste2001,Kurdyka2014-1,Tibar1999}. So far, the characterization of bifurcation values at infinity of polynomials in several variables ($n>2$) is still an open challenging problem. Up to now, most of the studies carried out in this direction require extra conditions. In the complex case, the result of~\cite{HaHV1984} was generalized by Parusi{\'n}ski under the assumption of isolated singularities at infinity~\cite{Parusinski1995}, then by Siersma and Tib\u{a}r for complex polynomials with isolated $W$-singularities at infinity~\cite{Siersma1995} and by~\cite{HaHV2011-1} for polynomial mappings of one dimensional fibers. In the real case, some sufficient condition for the existence of vanishing components at infinity for polynomial functions are given in~\cite{Dinh2012,Dinh2013}.  Moreover, when the fibers of a polynomial mapping are real curves, bifurcation values of $f$ can also be characterized \cite{Joita2017}.

In general, it is not easy to check if a value is a bifurcation value at infinity or not. People usually consider a larger but finite set which contains $B_\infty(f)$~\cite{HaHV2008-5,Kurdyka2000-1,Parusinski1997}, the set $K_\infty(f)$ of asymptotic critical values of $f$, to control $B_\infty(f)$. The set of {\em asymptotic critical values} of $f$ is defined as follows
$$K_\infty(f) := \left\{\begin{array}{lllll}
y \in {\mathbb K}& : & \textrm{there exists a sequence } x^k \in \mathbb{K}^n \textrm{ such that}\\
&& \|x^k\| \rightarrow +\infty, \ f(x^k) \rightarrow y \text{ and } \|x^k\| \| \nabla f(x^k) \| \to 0
\end{array}\right\}.$$
It is the set of points where Malgrange's condition fails to hold.
Generally, checking if a value is an asymptotic critical value is easier than checking if it is a bifurcation value at infinity since the set of {\em generalized critical values} $K(f)=K_0(f)\cup K_\infty(f)$ can be effectively computed \cite{Dinh2021-1,Jelonek2003-1,Jelonek2003,Jelonek2005,Jelonek2014,Jelonek2017}.

Now let $f$ be a real polynomial function in $n$ variables. We gives some necessary conditions of equisingularity at infinity for the family of the fibers of $f$ by studying the variation of the set of tangent directions $D_\infty(t)$ at infinity of the fibers of $f$ (Definition~\ref{AlgebraicTangent}). In fact, we will consider the two following functions $t\mapsto D_\infty(t)$ and $t\mapsto \vol_{n-2}(D_\infty(t)).$ The first function is the multivalued function which associates to each value $t\in\R$ the set of tangent directions $D_\infty(t)$ at infinity of the fiber $f^{-1}(t)$ and the second one is the composition of the $(n-2)$-dimensional volume function with the first one. It turns out that these functions are locally Lipschitz  outside $K_\infty(f)$ and bad behaviors only occur when $t$ is an asymptotic critical value. Precisely, the main results of the paper are the following.

\medskip\noindent \textbf{Theorem~\ref{LocallyLipschitzHausdorff}.} 
{\em Assume that $t_0\not\in K_\infty(f)$, where $f \colon \mathbb{R}^n \rightarrow \mathbb{R}$ is a polynomial function with $n\geqslant 2$. Then there exist some constants $c>0$ and $\delta>0$ such that for all $t_1,t_2\in(t_0-\delta,t_0+\delta)$, we have 
$$
\dist^{g}_{H}(D_\infty({t_1}),D_\infty({t_2}))\leqslant c|t_1-t_2|,
$$
where $\dist_{H}^g(.,.)$ denotes the Hausdorff distance with respect to the intrinsic metric in $\mathbf{D}_\infty^a$ and $\mathbf{D}_\infty^a$ is the set of algebraic tangent directions at infinity of $f$ defined by~\eqref{AT}.}

\medskip

\medskip\noindent \textbf{Theorem~\ref{VolumeContinuous}.} 
{\em Let $f \colon \mathbb{R}^n \rightarrow \mathbb{R}$ be a polynomial function  with $n\geqslant 2$. Then the volume function $t\mapsto \vol_{n-2}(D_\infty(t))$ is locally Lipschitz on $\R\setminus K_\infty(f).$}\medskip

Although our results are inspired more or less by the results given in \cite{Dutertre2019,Grandjean2008,Grandjean2009}, which study the variation of the total curvature and the total absolute curvature of the fibers of real polynomial or definable functions, our approach is somehow different. As the total curvature of the fibers is related to their topology by the Gauss-Bonnet-Chern Theorem, it seems that the set of tangent directions at infinity of the fibers and their $(n-2)$-dimensional volume are more related to the geometry at infinity of the fibers. Therefore, this gives a different point of view for the problem of studying singularities at infinity of real polynomial functions.

The paper is structured as follows. In Section~\ref{semi-algebraicGeometry}, we will recall some known results of Semi-Algebraic Geometry. The notions of geometric and algebraic tangent cones at infinity, their corresponding sets of tangent directions at infinity and some basic properties of these sets are given in Section~\ref{AGTangent}. Section~\ref{Main} contains the results on Lipschitz continuity of the sets of tangent directions at infinity and their volume. 

\section{semi-algebraic geometry}\label{semi-algebraicGeometry}

\subsection{Notation}

Let us start with some notation which will be used consistently throughout the paper. Let $\mathbb B^n_r(x)$ and $\mathbb S^{n-1}_r(x)$ denote, respectively, the open ball and the sphere of radius $r$ centered at $x$ in $\mathbb R^n$. For simplicity, we write $\mathbb B^n_r$ and $\mathbb S^{n-1}_r$ if $x=0$; and write $\mathbb B^n$ and $\mathbb S^{n-1}$ if $x=0$ and $r=1$. 

Let $X$ be a subset of $\mathbb R^n$.  The closure and the boundary of $X\subset\mathbb R^n$ is denoted by $\overline X$ and $\partial X$ respectively. 
Designate by  $\sing(X)$ the set of singular points of $X$, which is the set of points where $X$ is not a $C^1$-manifold.

We denote by $\dist(\cdot, \cdot)$ the Euclidean distance on $\mathbb R^n$ and set 
$$\mathscr{N}_r(X):=\{x\in\R^n :\ \dist(x,X)\leqslant r\},$$ 
the closed neighborhood of radius $r$ of $X$ in $\R^n$. 

Let $\dist^{X}(\cdot, \cdot)$ be the intrinsic metric in $X$.  The Hausdorff distance on $\mathbb R^n$ and the Hausdorff distance with respect to the intrinsic metric in $X$ are denoted respectively by $\dist_H(\cdot, \cdot)$ and $\dist_H^g(\cdot, \cdot).$


\subsection{Definition and basic properties}

In this part, we recall some notions and basic results of Semi-Algebraic Geometry, which can be found in~\cite{Benedetti1990, Bierstone1988, Bochnak1998, Yomdin2004}.

\begin{definition}{\rm
\begin{enumerate} [{\rm (i)}]
\item A set $X \subset \R^n$ is said to be {\em semi-algebraic} if it can be represented in a form $X = \bigcup_{i=1}^pX_i,$ with $X_i=\bigcap_{j=1}^{j_i}X_{ij}$, where each $X_{ij}$ has one of the following forms
$$\{x \in\R^n : \ f_{ij}(x) = 0\}, \quad \{x \in\R^n : \ f_{ij}(x) > 0\}$$
and each $f_{ij}$ is a polynomial (of degree $d_{ij}$). Clearly a representation of $X$ in the above form is not unique.

\item The set of data: 
$$\left(n,p,j_1,\ldots,j_p,(d_{ij})_{\substack{i=1,\ldots,p\\ j=1,\ldots,j_i}}\right)$$ 
is called the {\em diagram} $\mathscr{D}(X)$ of (the representation of) $X$. 

\item Let $X \subset \mathbb{R}^n$ and $Y \subset \mathbb{R}^p$ be semi-algebraic sets. A mapping $f \colon X \to Y$ is said to be {\em semi-algebraic} if its graph
$$\{(x, y) \in X \times Y : \ y = f(x)\}$$
is a semi-algebraic subset of $\mathbb{R}^n\times\mathbb{R}^p.$
\end{enumerate}
}\end{definition}

An important fact of Semi-Algebraic Geometry is Tarski--Seidenberg Theorem~\cite{Benedetti1990, Bierstone1988, Bochnak1998, Seidenberg1954, Tarski1931, Tarski1951, Yomdin2004}.

\begin{theorem}[Tarski--Seidenberg Theorem - first form] 
Let $X\subset\mathbb R^n$ be a semi-algebraic set. Then the image of $X$ by a semi-algebraic mapping is semi-algebraic. Moreover, its diagram depends only on the diagram of $X.$
\end{theorem}

An equivalent version of Tarski--Seidenberg Theorem, that we give below, asserts that semi-algebraic sets can be expressed using formulas containing quantifiers. Let us first specify the notion of {\em first-order formula}:
\begin{enumerate} [{\rm 1.}]
\item If $P\in\mathbb{R}^n[x_1,\ldots,x_n]$, then $P=0$ and $P>0$ are first-order formulas.
\item If $\Phi$ and $\Psi$ are first-order formulas, then ``$\Phi\vee\Psi$", ``$\Phi\wedge\Psi$", ``$\neg\Phi$" 
are first-order formulas.
\item If $\Phi$ is a first-order formula and $x$ is a variable ranging over $\mathbb{R}$, then $\exists x\Phi$ and $\forall x\Phi$ are first-order formulas.
\end{enumerate}

The formulas obtained by using only rules $1$ and $2$ are called {\em quantifier-free formulas}. By definition, a subset $A\subset\mathbb{R}^n$ is semi-algebraic if and only if there is a quantifier-free formula $\Phi(x_1,\ldots,x_n)$ such that 
$$(x_1,\ldots,x_n)\in A\leftrightarrow \Phi(x_1,\ldots,x_n).$$

\begin{theorem}[Tarski--Seidenberg Theorem - second form] \label{TarskiSeidenberg2}
For any first-order formula $\Phi(x_1,\ldots,x_n)$, the set of $(x_1,\ldots,x_n)\in\mathbb{R}^n$ which satisfies $\Phi(x_1,\ldots,x_n)$ is semi-algebraic. In other words, every first-order formula is equivalent to a quantifier-free formula.
\end{theorem}

We list below some basic properties of semi-algebraic sets and mappings:
\begin{enumerate} [{\rm (i)}]
\item The class of semi-algebraic sets is closed with respect to Boolean operators; a Cartesian product of semi-algebraic sets is a semi-algebraic set.

\item The closure, the interior, the boundary and each connected components of a semi-algebraic set $X\subset\mathbb R^n$ are semi-algebraic sets. Furthermore, the diagram of these sets depends only on the diagram of $X.$

\item The set of singular points of a semi-algebraic set $X\subset\mathbb R^n$ is semi-algebraic.

\item  A composition of semi-algebraic mappings is a semi-algebraic mapping.

\item  If $X$ and $Y\ne\emptyset$ are semi-algebraic sets, then the distance function $$\dist(\cdot, Y) \colon X \to {\mathbb R}, \quad x \mapsto \mathrm{dist}(x, Y) := \inf \{\|x - a\| : \ a \in Y \},$$
is continuous semi-algebraic.
\end{enumerate}

The following Curve Selection Lemma will be useful later \cite{Milnor1968,Nemethi1992}.

 \begin{lemma}\label{CSL}
Let $X\subset \mathbb{R}^n$ be a semi-algebraic set.
Assume that there exists a sequence $x^k\in X$ such that $\displaystyle\lim_{k \to \infty} x^k  = x\in\overline X\setminus X.$ 
Then there exists a $C^1$ semi-algebraic curve $\varphi \colon (0, \epsilon)\to X\setminus \{x\}$ such that $\displaystyle\lim_{t \to 0} \varphi(t) = x.$
 \end{lemma}

Now we recall Hardt's semi-algebraic local trivialization theorem.
\begin{theorem}\label{HardtTheorem}
Let $X$ and $Y$ be respectively semi-algebraic sets in $\mathbb{R}^n$ and $\mathbb{R}^m$, $f \colon X \rightarrow Y$ a continuous semi-algebraic mapping.
Then there exists a partition of $Y$ into finitely many semi-algebraic subsets $Y_i, i = 1, \ldots p,$ such that $f$ is semi-algebraically trivial over each $Y_i,$ i.e., $f^{-1}(Y_i)$ is semi-algebraically homeomorphic to $f^{-1}(y_i)\times Y_i$ for each $i$ and any $y_i \in Y_i.$
\end{theorem}

\subsection{Stratification and Whitney property of semi-algebraic sets}
Let $X\subseteq\R^n$ be a semi-algebraic set.

A {\em semi-algebraic stratification} of $X$ is a partition of $X$ into a locally finite family $\SA$ of connected semi-algebraic $C^1$-submanifolds of $\R^n$ such that the following {\em frontier condition} is satisfied: if $Y\cap (\overline Z\setminus Z)\ne\emptyset$ for $Y,Z\in\SA$, then $Y\subset (\overline Z\setminus Z)$ and $\dim Y<\dim Z.$

According to~\cite[Proposition 2.5.1]{Benedetti1990}, every semi-algebraic set admits a semi-algebraic stratification. 
Moreover, the number of strata and their diagrams depending only on the diagram of $X$ in view of~\cite[Proposition 4.4]{Yomdin2004}.

Let $X\subset\mathbb R^n$ be a semi-algebraic set and let $\mathcal{S} := \{X_\alpha\}_{\alpha \in I}$ be a semi-algebraic stratification of $X.$ The {\em dimension} of $X$ is defined by
$$\dim X:=\max\{\dim X_\alpha:\ \alpha\in I\}.$$
It is not hard to verify that this definition of dimension does not depend on the stratification of $X.$
For convenience, set $\dim\emptyset=-1$. 
Let $x\in X$, the {\em dimension of $X$ at $x$} is defined by
$$\dim_x X:=\max\{\dim X_\alpha:\ \alpha\in I,\ x\in\overline{X}_\alpha\}.$$ 
Obviously $\dim_x X=\dim T_xX$ if $x$ is a non singular point of $X$, where $T_xX$ denotes the tangent space of $X$ at $x$. 


We say that $X$ has the {\it Whitney property} if for any $a\in X$, there exists a neighborhood $U$ of $a$ and two constants $M>0$ and $\alpha>0$ such that any points $x$ and $y$ in $X\cap U$ can be joined in $X\cap U$ by a piecewise smooth curve of length $\leqslant M\|x - y\|^\alpha.$ In view of~\cite{Stasica1982}, if, in addition, $X$ is closed, then $X$ has the {\it Whitney property}. Although the constants $M$ and $\alpha$ depend on $U$, if $X$ is connected and compact, we can choose $U=X$ for any $a\in X$, which means $M$ and $\alpha$ depend only on $X$ (cf. \cite{Kurdyka1997}). In this case we say that $X$ has the {\it Whitney property with constant $M$ and exponent $\alpha$}. On the other hand, in light of~\cite{Kurdyka1991,Kurdyka1997}, for any constant $M>1$, there exists a semi-algebraic stratification $\SA$ of $X$ such that each stratum $Y\in\SA$ has the Whitney property with constant $M$ and exponent $1$. In the following result, which is crucial in the proof of Theorem~\ref{LocallyLipschitzHausdorff}, we strengthen this statement by claiming that it still holds on the closure of each stratum $Y\in\SA$.

\begin{proposition}\label{WhitneyProperty} 
Let $M>1$ and let $X\subseteq\R^n$ be a semi-algebraic set, then there exists a semi-algebraic stratification $\SA$ of $X$ such that for each stratum $Y\in\SA$, any two points $x,y\in \overline Y$ can be joined in $\overline Y$ by a piecewise smooth arc of length $\leqslant M\|x - y\|.$ In particular $\dist^g(x,y)\leqslant M\|x-y\|.$
\end{proposition}
\begin{proof} 
Let $\SA$ be a semi-algebraic stratification of $X$ such that each stratum in $\SA$ has the Whitney property with constant $\displaystyle\frac{M+1}{2}$ and exponent $1$ in view of~\cite{Kurdyka1991,Kurdyka1997}. 
Let $Y\in\SA$ and pick two points $x,y\in \overline Y$ arbitrarily. 
Clearly, we may suppose that $x\ne y$. In view of~\cite{Stasica1982}, there exists a neighborhood $U$ (resp., $V$) of $x$ (resp., $y$) and some positive constants $\alpha_1, M_1$ (resp., $\alpha_2, M_2$) such that $\overline Y\cap U$ (resp., $\overline Y\cap V$) has the Whitney property with constant $M_1$ (resp., $M_2$) and exponent $\alpha_1$ (resp., $\alpha_2$). Let 
$$\widetilde M:=\max\left\{M_1,M_2,\frac{M+1}{2}\right\}\ \text{ and }\ \alpha:=\min\{\alpha_1,\alpha_2,1\}>0.$$ 
Observe that we can pick two points $x'\in Y\cap U$ and $y'\in  Y\cap V$ arbitrarily close to $x$ and $y$ respectively so that 
$$\max\{\|x-x'\|,\|y-y'\|\}\leqslant 1 \ \text{ and }\ \max\{\|x-x'\|^\alpha,\|y-y'\|^\alpha\}\leqslant\frac{M-1}{8\widetilde M}\|x-y\|.$$ 
Then $x$ and $x'$ can be joined in $\overline Y\cap U$ by a piecewise smooth curve of length bounded by
$$M_1\|x - x'\|^{\alpha_1}\leqslant \widetilde M\|x - x'\|^{\alpha}\leqslant\frac{M-1}{8}\|x-y\|.$$
Similarly  $y$ and $y'$ can be also joined in $\overline Y\cap V$ by a piecewise smooth curve of length bounded by $\displaystyle\frac{M-1}{8}\|x-y\|.$
Moreover, by the construction, $x'$ and $y'$ can be joined in $Y$ by a piecewise smooth curve of length bounded by $\displaystyle\frac{M+1}{2}\|x' - y'\|.$ Summarily, $x$ and $y$ can be joined in $\overline Y$ by a piecewise smooth curve of length bounded by
\begin{eqnarray*} 
&&\frac{M-1}{4}\|x-y\|+\frac{M+1}{2}\|x' - y'\|\\
&\leqslant &\frac{M-1}{4}\|x-y\|+\frac{M+1}{2}\left(\|x - y\|+\|x-x'\|+\|y-y'\|\right)\\
&\leqslant &\frac{M-1}{4}\|x-y\|+\frac{M+1}{2}\|x - y\|+\widetilde M\left(\|x-x'\|^\alpha+\|y-y'\|^\alpha\right)\\
&\leqslant &\frac{M-1}{4}\|x-y\|+\frac{M+1}{2}\|x - y\|+\frac{M-1}{4}\|x-y\|=M\|x-y\|.
\end{eqnarray*}
The proposition is proved.
\end{proof}

Note that Proposition~\ref{WhitneyProperty} can be stated for a more general context of subanalytic sets like those given in~\cite{Kurdyka1991,Kurdyka1997,Stasica1982} but we only consider the semi-algebraic case which is totally enough for our purposes. 

\subsection{Some variational properties of semi-algebraic sets} 
In this subsection, we recall briefly some notions and results on variations of semi-algebraic sets which will be used afterwards. The contents presented in this part can be found in~\cite{Yomdin2004} and the references therein. 

\begin{definition}
Let $X\subset\R^n$ be a bounded set. For any $\epsilon > 0,$ denote by $M(\epsilon,X)$ the minimal number of closed balls of radius $\epsilon$ that cover $X$. 
The real number $\log_2 M(\epsilon,X)$ is called the {\em $\epsilon$-entropy} of $X$.
\end{definition}

Denote by $G_n^k$ and $\bar G_n^k$ the space of all the $k$-dimensional linear subspaces and the space of all the $k$-dimensional affine subspaces in $\R^n$ respectively. Each element $\bar P$ in $\bar G_n^{n-k}$ can be represented by a pair $(x,P)\in \R^n\times G_n^{n-k}$ where $x\in P$ and $\bar P = \bar P_x$ is the $k$-dimensional affine subspace of $\R^n$, orthogonal to $P$ at $x$. Let $d\bar P$ be the measure on $\bar G_n^{n-k}$ given by $d\bar P=dx\otimes dP$ where $dx$ is the Lebesgue measure on $P$ (identify $P$ with $\R^{n-k}$) and $dP$ is the measure on $G_n^k$ induced by the Haar measure on the orthogonal group $\mathcal O_n(\R)$ of $\mathbb{R}^n.$

\begin{definition}
Let $X$ be a bounded subset of $\R^n$. Define $V_0(X)$ as the number of connected components of $X$. For $i= 1,\ldots,n$, the {\em $i$-th variation} of $X$, denoted by $V_i(X)$, is defined as follows:
$$V_i(X) = c(n, i) \int_{\bar P\in \bar G_n^{n-i}}V_0(X\cap\bar P)d\bar P,$$
where the coefficient $c(n,i)$ is chosen in such a way that $V_i([0,1]^i) = 1.$
\end{definition}

\begin{proposition}[{see \cite[Proposition 5.8]{Yomdin2004}}]\label{5.8} 
Let $X\subset\overline\B_1^n\subset\R^n$ be a semi-algebraic set of dimension $l$, $Y\subset\R^n$ and $0 < \eta < 1$ such that $Y\subset \mathscr{N}_\eta(X)$. Assume that $X$ has the diagram $$\mathscr{D}(X)=\left(n,p,j_1,\ldots,j_p,(d_{ij})_{\substack{i=1,\ldots,p\\ j=1,\ldots,j_i}}\right).$$
Then we have
$$M(\eta,Y) \leqslant\left(\frac{4}{\eta}\right)^l C(n)\nu(l)\alpha(n),$$
where $C(n)$ is a positive constant depending only on $n$, 
$$\displaystyle\nu(l)=\sum_{i=0}^l c(n,i)\vol_i(\B_1^i)\ \text{ and }\ \displaystyle\alpha(n)=\frac{1}{2}\sum_{i=1}^p(d_i + 2)(d_i + 1)^{n-1}$$ 
with $d_i=\displaystyle\sum_{j=1}^{j_i}d_{ij}.$
\end{proposition}

\begin{proposition}[{see \cite[Theorem 5.14]{Yomdin2004}}] \label{5.14} 
Let $X$ be a bounded semi-algebraic set of dimension $l.$ Then for any $\epsilon > 0,$ we have
$$c_1\sum_{i=0}^l V_i(X)\left(\frac{1}{\epsilon}\right)^i\leqslant M(\epsilon,X)\leqslant c_2\sum_{i=0}^l V_i(X)\left(\frac{1}{\epsilon}\right)^i,$$
where $c_1$ and $c_2$ are positive constants depending only on the diagram $\mathscr{D}(X)$ of $X.$
\end{proposition}

\subsection{Semi-algebraic multivalued functions}

As we will consider the tangent cones at infinity and the sets of tangent directions at infinity of the fibers of semi-algebraic functions afterwards, when the value of the functions varies, so do the tangent cones at infinity and the sets of tangent directions at infinity, this means that we need to deal with semi-algebraic multivalued functions. Like semi-algebraic single-valued mappings, the definition of a semi-algebraic multivalued function is based on the semi-algebraic property of its graph.

\begin{definition}{\rm 
Let $X \subset \mathbb{R}^n$ and $Y \subset \mathbb{R}^p$ be semi-algebraic sets. A multivalued function $\mathcal{F} \colon X \rightrightarrows Y$ is said to be {\em semi-algebraic} if its graph
$$\{(x, y) \in X \times Y : \ y \in \mathcal{F}(x)\}$$
is a semi-algebraic subset of $\mathbb{R}^n\times\mathbb{R}^p.$
}\end{definition}

Let us next define the locally Lipschitz continuity of a multivalued function.
\begin{definition}
Let $\mathcal{F}\colon X \rightrightarrows Y$ be a multivalued function where $X \subset \mathbb{R}^n$ and $Y \subset \mathbb{R}^p.$ We say that $\mathcal{F}$ is {\em locally Lipschitz} at $x\in X$ if there exist some constants $c>0$ and $\delta>0$ such that for all $t_1,t_2\in\B_\delta^n(x)$, we have 
$$\mathcal{F}(t_1) \subset \mathcal{F}(t_2) + c|t_1 - t_2|\B^p.$$
\end{definition}


\section{Tangent directions at infinity}\label{AGTangent}

Let $f \colon \mathbb{R}^n \rightarrow \mathbb{R}$ be a polynomial function of degree $d \ge 1$  with $n\geqslant 2.$ We can write
$$f(x) = f_d(x) + f_{d - 1}(x) + \cdots,$$ 
where $f_i$ is the homogeneous part of degree $i$ of $f.$ 

Put
\begin{equation}\label{AT}\mathbf{D}_\infty^a  := \left\{u\in\mathbb S^{n-1} : \ f_d(u) = 0 \right \}\end{equation}
and call it the {\em set of algebraic tangent directions at infinity} of fibers of $f.$ The following notion plays a crucial role in this paper.

\begin{definition}\label{AlgebraicTangent}{\rm 
For each $t \in \mathbb{R},$  the {\em set of geometric tangent directions at infinity} of the fiber $f^{-1}(t)$ is defined by
$$D_\infty(t) := \left\{u\in\mathbb S^{n-1} : \ \text{there is a sequence } x^k\in f^{-1}(t) \text{ such that } x^k\to\infty \text{ and } \frac{x^k}{\|x^k\|}\to u\right\}.$$
}\end{definition}

Some simple properties of tangent directions at infinity are given below.
\begin{lemma}\label{GeometricAlgebraic}
\begin{enumerate}[{\rm (i)}] 
\item $\mathbf{D}_\infty^a $ is an algebraic set of dimension at most $n - 2.$

\item  For all $t \in \R$, $D_\infty(t)$ is a semi-algebraic subset of $\mathbf{D}_\infty^a $ and it holds that
$$dim D_\infty(t) \le \dim f^{-1}(t) - 1 \le n - 2.$$

\item There is a representation of $D_\infty({t})$ such that the diagram of $D_\infty({t})$ depends only on the dimension $n$ and the degree of $f.$

\item For each $t\in\mathbb R$, let 
$$X_{t}:= \{x\in D_\infty({t})\setminus(\{\nabla f_d=0\}\cup \sing(D_\infty({t}))):\ \dim_x D_\infty(t)=n-2\}.$$
Then $X_t$ is semi-algebraic. Furthermore, there exists a representation of $X_t$ such that the diagram of $X_t$ depends only on the dimension $n$ and the degree of $f.$

\end{enumerate}
\end{lemma}

\begin{proof}
(i): This is clear. 

\noindent
(ii) and (iii): Take any $u \in D_\infty(t).$ By definition, there is a sequence $x^k\in f^{-1}(t)$ such that $x^k\to\infty$ and $\displaystyle\frac{x^k}{\|x^k\|}\to u.$ 
Observe that
$$f_d\left(\frac{x^k}{\|x^k\|}\right) = \frac{f_d(x^k)}{\|x^k\|^d} = \frac{f(x^k) - \sum_{i=0}^{d-1}f_i(x^k)}{\|x^k\|^d} = \frac{t}{\|x^k\|^d} - \sum_{i=0}^{d-1}\frac{1}{\|x^k\|^{d-i}}f_i\left(\frac{x^k}{\|x^k\|}\right).$$
Letting $k \to \infty,$ we get $f_d(u)=0,$ and so $u \in \mathbf{D}_\infty^a .$ Therefore $D_\infty(t) \subseteq \mathbf{D}_\infty^a .$ 

For each $t \in \mathbb{R},$ define
\begin{equation}\label{At}A(t) :=\{(\lambda x, \lambda) \in \mathbb{R}^{n} \times (0,+\infty)  : \ f(x) = t\}.\end{equation}
Clearly, $A(t)$ is a semi-algebraic set, which is homeomorphic to $f^{-1}(t) \times (0,+\infty).$ Hence
$$\dim A(t) = \dim f^{-1}(t) + 1.$$
Note that 
$$\overline{A}(t) \cap (\mathbb{R}^{n} \times \{0\}) \subset \overline{A}(t) \setminus A(t).$$ 
By \cite[Proposition~1.4]{HaHV2017}, therefore
$$\dim \overline{A}(t) \cap (\mathbb{R}^{n} \times \{0\}) \leqslant \dim \left (\overline{A}(t) \setminus A(t) \right) < \dim A(t) = \dim f^{-1}(t) + 1.$$
On the other hand, it is clear that $\overline{A}(t) \cap (\mathbb{R}^{n} \times \{0\})$ is the cone with the apex at the origin and the base being $D_\infty(t)\times\{0\}.$ Therefore, $\dim D_\infty(t) \leqslant f^{-1}(t) - 1.$

Finally, observe that the diagram of $A(t)$ depends only on the dimension $n$ and the degree of $f,$ so does 
$\overline{A}(t) \cap (\mathbb{R}^{n} \times \{0\}).$ Consequently, the diagram of $D_\infty(t)$ depends also only 
on the dimension $n$ and the degree of $f.$ 

(iv) The first statement is clear so let us prove the second one. First of all, we show that 
\begin{equation}\label{Xt1}X_t= D_\infty({t})\setminus(\{\nabla f_d=0\}\cup \overline{\mathbf{D}_\infty^a\setminus D_\infty(t)}).\end{equation}

Pick arbitrarily $u\in D_\infty({t})\setminus(\{\nabla f_d=0\}\cup \overline{\mathbf{D}_\infty^a\setminus D_\infty(t)}).$ 
Then there is $r>0$ such that $\mathbb B^n_r(u)\cap \mathbf{D}_\infty^a\setminus D_\infty(t)=\emptyset,$ i.e., 
\begin{equation}\label{BD}\mathbb B^n_r(u)\cap \mathbf{D}_\infty^a=\mathbb B^n_r(u)\cap D_\infty(t).\end{equation}
As $u\notin \{\nabla f_d=0\}$, it is not a singular point of $\mathbf{D}_\infty^a$ and $\dim_u \mathbf{D}_\infty^a=n-2.$
By combining these with~\eqref{BD}, it follows that $u\notin \sing(D_\infty(t))$ and $\dim_u D_\infty(t)=n-2.$
Consequently, $u\in X_t$ and so
$$X_t\supseteq D_\infty({t})\setminus(\{\nabla f_d=0\}\cup \overline{\mathbf{D}_\infty^a\setminus D_\infty(t)}).$$

Now for any $u\in X_t$, we will show that $u\not\in \overline{\mathbf{D}_\infty^a\setminus D_\infty(t)}),$ which implies 
$$X_t\subseteq D_\infty({t})\setminus(\{\nabla f_d=0\}\cup \overline{\mathbf{D}_\infty^a\setminus D_\infty(t)}),$$
and so yields~\eqref{Xt1}. Assume for contradiction that $u\in \overline{\mathbf{D}_\infty^a\setminus D_\infty(t)}).$ 
Then by Lemma~\ref{CSL}, there is a $C^1$ semi-algebraic curve 
$$\varphi \colon (0, \epsilon)\to \mathbf{D}_\infty^a\setminus D_\infty(t)$$ 
such that $\displaystyle\lim_{t \to 0} \varphi(t) = u.$
On the other hand, since $u\not\in\sing(D_\infty(t)),$ for any sufficiently small neighborhood $U$ of $u$, the restriction $\pi|_{U\cap D_\infty({t})}$ is one-to-one, where 
$$\pi\colon\mathbb R^n\to \{u\}+T_u D_\infty({t})$$ 
is the orthogonal projection on the tangent plane $\{u\}+T_u D_\infty({t}).$ 
It is clear that $\pi|_{U\cap \mathbf D_\infty^a}$ is not one-to-one.
Hence $u\in \sing(\mathbf{D}_\infty^a),$ i.e., $\nabla f_d(u)=0$, which is a contradiction. 
Consequently~\eqref{Xt1} follows. 
Observe that $$\widetilde\pi(\overline{A}(t) \cap (\mathbb{S}^{n-1} \times \{0\}))=D_\infty(t),$$ 
where $\widetilde\pi\colon\mathbb R^{n+1}\to\mathbb R^n$ is the projection on the first $n$ coordinates and $A(t)$ is given by~\eqref{At}. 
Hence, in view of~\eqref{Xt1}, we have
$$X_t= \widetilde\pi(\overline{A}(t) \cap (\mathbb{S}^{n-1} \times \{0\}))\setminus(\{\nabla f_d=0\}\cup \overline{\mathbf{D}_\infty^a\setminus \widetilde\pi(\overline{A}(t) \cap (\mathbb{S}^{n-1} \times \{0\}))}).$$
As the diagrams of $A(t),\ \{\nabla f_d=0\}$ and $\mathbf{D}_\infty^a$ depend only on the dimension $n$ and the degree of $f,$ so does the diagram of $X_t$. Thus item (iv) follows.
\end{proof}


We next provide some examples and remarks concerning tangent directions at infinity.

\begin{example}
Consider the polynomial function 
$$f \colon \mathbb{R}^3 \to \mathbb{R}, \quad (x, y, z) \mapsto f(x,y,z)=z-x^2-y^2.$$
Some simple computations show that $K_\infty(f) = \emptyset,$ $\mathbf{D}_\infty^a  = \{(0, 0, \pm 1)\},$ and that
\begin{eqnarray*}
 D_\infty(t) \ = \  \{(0, 0, 1)\} \quad  \textrm{  for all  } \quad t \in \mathbb{R};
\end{eqnarray*}
in particular, $\dim D_\infty(t) = 0 < 2 = \dim f^{-1}(t).$
\end{example}

The following example shows that in general, we can not expect that the multivalued function 
$${D}_\infty \colon \mathbb{R} \rightrightarrows \mathbb{S}^{n - 1}, \quad t \mapsto D_\infty(t),$$ 
is constant over each connected component of $\mathbb{R} \setminus K_\infty(f).$ 

\begin{example}
Consider the following polynomial (see \cite{Parusinski1997} and \cite[Example 2.1]{Kurdyka2000-1})
$$f(x,y,z) = x + x^2y + x^4yz.$$ 
We have $K_\infty(f) = \{0\}$ and $\mathbf{D}_\infty^a =\{(x,y,z)\in\mathbb S^2:\ xyz=0\}.$ Let us prove that the multivalued function
$$t\mapsto H_t:=D_\infty(t)\cap \{x=0\}=D_\infty(t)\cap \{(0,y,z):\ y^2+z^2=1\}$$ 
is not constant on the intervals $(-\infty, 0)$ and $(0, +\infty).$ To do this, for any sequence $(x_k,y_k,z_k)\in f^{-1}(t)$ tending to infinity, we need to investigate the cluster points of the sequence $\displaystyle\frac{(x_k,y_k,z_k)}{\|(x_k,y_k,z_k)\|}$ belonging to $\{(0,y,z):\ y^2+z^2=1\}$.
Without loss of generality, assume that the sequence $\displaystyle\frac{(x_k,y_k,z_k)}{\|(x_k,y_k,z_k)\|}$ converges to a limit $v=(0,y,z)\in\{(0,y,z):\ y^2+z^2=1\}$. 

Let $\displaystyle(x_k,y_k,z_k)=\left(t,\pm k,-\frac{1}{t^2}\right)\in f^{-1}(t)$, then $$\frac{(x_k,y_k,z_k)}{\|(x_k,y_k,z_k)\|}\to (0,\pm 1,0)\in H_t.$$ 
Furthermore, if
$(x_k,y_k,z_k)=(t,0,\pm k)\in f^{-1}(t)$, then $$\frac{(x_k,y_k,z_k)}{\|(x_k,y_k,z_k)\|}\to (0,0,\pm 1)\in H_t.$$

Now assume that $y_k \simeq z_k$ as $k\to+\infty.$ So set $z_k=\lambda_k y_k$. 
Consider the equation
\begin{equation}\label{fk}
\lambda_k x_k^4y_k^2+x_k^2y_k+x_k-t=0
\end{equation}
with $y_k$ as variable. Observe that $y_k,z_k\to\infty,$ $y,z\ne 0$ and $\displaystyle\lambda_k\to\frac{z}{y}\ne 0$.
These, together with~\eqref{fk} and the fact that $\displaystyle\frac{(x_k,y_k,z_k)}{\|(x_k,y_k,z_k)\|}\to v=(0,y,z)$, imply that $x_k\to 0.$ 

The discriminant of~\eqref{fk} is
$$\Delta_k=x_k^4-4\lambda x_k^4(x_k-t)=x_k^4(1+4\lambda_k t- 4\lambda_k x_k).$$ 
So~\eqref{fk} has real roots if and only if $1+4\lambda_k (t-x_k)\geqslant 0.$ Letting $k\to+\infty$, we have
\begin{equation}\label{Delta>0}
1+4\frac{z}{y} t\geqslant 0.
\end{equation} 
Consider two cases:

\subsubsection*{Case $t>0$}
The condition \eqref{Delta>0} becomes $\displaystyle\frac{z}{y}\geqslant-\frac{1}{4t}.$ Set $A:=\displaystyle\frac{(0,-1,\frac{1}{4t})}{\sqrt{1+\frac{1}{16t^2}}}.$
Then, on the unit circle centered at the origin in the plane $Oyz$, $H_t$ is the union of two anticlockwise arcs (see Figure 1)
$$\overgroup{A,(0,0,-1)}\quad \text{ and }\quad  \overgroup{-A,(0,0,1)}.$$

\subsubsection*{Case $t < 0$} 
The condition \eqref{Delta>0} becomes $\displaystyle\frac{z}{y}\leqslant-\frac{1}{4t}.$ Set $B:=\displaystyle\frac{(0,1,\frac{1}{4t})}{\sqrt{1+\frac{1}{16t^2}}}.$ 
Hence, on the unit circle centered at the origin in the plane $Oyz$, $H_t$ is the union of two anticlockwise arcs  (see Figure 2) 
$$\overgroup{(0,0,-1),B}\quad \text{ and }\quad   \overgroup{(0,0,1),-B}.$$

\begin{tikzpicture}
\def\d{3};\def\k{0.5};\def\r{2.5};\def\m{2};

\coordinate (O1) at (-\d-\k,0);
\coordinate (v) at ($\r/sqrt(2)*(-1,1)$);

\draw [->,line width=1pt](-2*\d-\k,0) -- (-\k,0);
\draw [->,line width=1pt](-\d-\k,-\d) -- (-\d-\k,\d);
\draw [domain=-2*\d: -1.5*\k,line width=1pt,variable=\t] plot ({\t},{-\t-\d-\k});
\draw [blue,line width=1pt] ($(O1)+(v)$) arc (135:270:\r);
\draw [blue,line width=1pt] ($(O1)-(v)$) arc (-45:90:\r);

\node [below] at (-\d-\k,-\d-2*\k)  {Figure 1: $t>0$.};
\node [right] at ($(O1)+(0,\k/\m)$) {$Ox$};
\node [above] at (-\k,0) {$y$};
\node [right] at (-\d-\k,\d) {$z$};
\node [left]  at ($(O1)+(v)$) {$A$};
\node [right] at ($(O1)-(v)$) {$-A$};
\node [below] at (-2*\k,-\d+\k) {$\displaystyle z=-\frac{y}{4t}$};

\filldraw (-\d-\k,0) circle (2pt);
\filldraw ($(O1)+(v)$) circle (2pt);
\filldraw ($(O1)-(v)$) circle (2pt);
\filldraw (-\d-\k,-\r) circle (2pt) node[right] {$(0,0,-1)$};
\filldraw (-\d-\k,\r) circle (2pt) node[left] {$(0,0,1)$};

\coordinate (O2) at (\d+\k,0);
\coordinate (w) at ($\r/sqrt(2)*(1,1)$);

\draw [->,line width=1pt](\k,0) -- (2*\d+\k,0);
\draw [->,line width=1pt](\d+\k,-\d) -- (\d+\k,\d);
\draw [domain=2*\k: 2*\d+0.5*\k,line width=1pt,variable=\t] plot ({\t},{\t-\d-\k});
\draw [blue,line width=1pt] ($(O2)+(0,\r)$) arc (90:225:\r);
\draw [blue,line width=1pt] ($(O2)-(0,\r)$) arc (-90:45:\r);

\node [below] at (\d+\k,-\d-2*\k)  {Figure 2: $t<0$.};
\node [right] at ($(O2)-(0,\k/\m)$) {$Ox$};
\node [above] at (2*\d+\k,0) {$y$};
\node [right] at (\d+\k,\d) {$z$};
\node [right] at ($(O2)+(w)$) {$B$};
\node [left ] at ($(O2)-(w)$) {$-B$};
\node [above] at (2*\d,\d-\k) {$\displaystyle z=-\frac{y}{4t}$};

\filldraw (\d+\k,0) circle (2pt);
\filldraw ($(O2)+(w)$) circle (2pt);
\filldraw ($(O2)-(w)$) circle (2pt);
\filldraw (\d+\k,-\r) circle (2pt) node[left] {$(0,0,-1)$};
\filldraw (\d+\k,\r)  circle (2pt) node[right] {$(0,0,1)$};
\end{tikzpicture}

On $\mathbb R\setminus\{0\}$, it is clear that the mapping $t\mapsto H_t$ is not constant whence neither is the mapping $t\mapsto D_\infty(t).$
\end{example}

The following example shows that, in general, the multivalued function 
$${D}_\infty \colon \mathbb{R} \rightrightarrows \mathbb{S}^{n - 1}, \quad t \mapsto D_\infty(t),$$ 
is not locally Lipschitz continuous.

\begin{example}
Consider the polynomial function (see~\cite{Dinh2013})
$$f \colon \mathbb{R}^3 \to \mathbb{R}, \quad (x, y, z) \mapsto f(x,y,z) = z(x^2+(xy-1)^2).$$ 
Some simple computations show that $D_\infty(0) = \{z=0,\ x^2+y^2=1\}$ and
$$D_\infty(t) = 
\begin{cases}
D_\infty(0)\cup\{x=0,y^2+z^2=1,\ z\geqslant 0\} & \text{ if } t>0,\\
D_\infty(0)\cup\{x=0,y^2+z^2=1,\ z\leqslant 0\} & \text{ if } t<0.
\end{cases}$$
Therefore, the multivalued function ${D}_\infty$ is not Lipschitz continuous around the value $t = 0.$ Observe that $0 \in K_\infty(f).$ (To see this, consider the sequence
$\displaystyle X^k := \left(\frac{1}{k},k,\frac{1}{k}\right)$ which tends to infinity as $k$ tends to infinity. Then it is easy to check that 
$$f(X^k) \to 0 \quad \textrm{ and } \quad \|X^k\| \|\nabla f(X^k)\| \to 0$$
as $k \to +\infty.$) On the other hand, in Corollary~\ref{LocallyLipschitz} below, it will be shown that the multivalued function ${D}_\infty$ is locally Lipschitz outside $K_\infty(f).$
\end{example}


\section{Continuity of the set of tangent directions at infinity and its volume}\label{Main}

Given a polynomial function $f \colon \mathbb{R}^n \rightarrow \mathbb{R},$ we would like to study the variation of the set of tangent directions at infinity of the fibers of $f$ and its volume while relating them to the set of asymptotic critical values of $f$. First of all, we need some preparation.

\begin{lemma}\label{length} 
Let $f \colon \mathbb{R}^n \rightarrow \mathbb{R}$ be a polynomial function and let $\gamma\colon [t_1,t_2]\to\mathbb R^n$ be an integral curve of the vector field $\displaystyle\frac{\nabla f}{\|\nabla f\|^2}$. Assume that there exists a constant $C>0$ such that 
$$\|\gamma(t)\|\|\nabla f(\gamma(t))\|\geqslant C \quad \textrm{ for all } \quad t \in [t_1, t_2].$$
Then 
$$\left\|\frac{\gamma(t_1)}{\|\gamma(t_1)\|}-\frac{\gamma(t_2)}{\|\gamma(t_2)\|}\right\| \leqslant \frac{2}{C} |t_1-t_2|.$$
\end{lemma}

\begin{proof} 
For each $t\in[t_1,t_2]$, let $\displaystyle\alpha(t):=\frac{\gamma(t)}{\|\gamma(t)\|}.$ We have
$$\begin{array}{lllll}
\|\alpha'(t)\|&=&\displaystyle\left\|\frac{(\gamma)'(t)\|\gamma(t)\|-\gamma(t)\|\gamma(t)\|'}{\|\gamma(t)\|^2}\right\|\\
&\leqslant&\displaystyle\frac{\|(\gamma)'(t)\|+\left\|\|\gamma(t)\|'\right\|}{\|\gamma(t)\|}\\
&\leqslant& \displaystyle 2\frac{\|(\gamma)'(t)\|}{\|\gamma(t)\|}=\displaystyle \frac{2}{\|\nabla(\gamma(t))\|\|\gamma(t)\|}\leqslant \frac{2}{C}.
\end{array}$$
Therefore
$$\left\|\frac{\gamma(t_1)}{\|\gamma(t_1)\|}-\frac{\gamma(t_2)}{\|\gamma(t_2)\|}\right\|=\|\alpha(t_1)-\alpha(t_2)\|\leqslant\int_{t_1}^{t_2}\|\alpha'(t)\|dt\leqslant \frac{2}{C} |t_1-t_2|.$$
\end{proof}

\begin{lemma}\label{NotFar} 
Let $f \colon \mathbb{R}^n \rightarrow \mathbb{R}$ be a polynomial function and let $I=(a,b)$ be an interval in $\mathbb R.$ Assume that there exist some constants $C > 0$ and $R > 0$ with $\displaystyle e^\frac{b-a}{C}<\frac{3}{2}$ such that 
\begin{equation}\label{C}
\|x\|\|\nabla f(x)\|\geqslant C \quad \text{ for  } \quad \|x\| \geqslant R \ \textrm{ and }  \ f(x)\in I.
\end{equation}
Let $t_1\in I$. Suppose that $x^k$ is a sequence in $f^{-1}({t_1}) \setminus\B^n_{2R}$ such that $x^k\to\infty$ and $\displaystyle\frac{x^k}{\|x^k\|}\to u.$ 
For each $k,$ let $\gamma^k(t)$ be the maximal integral curve of the vector field $\displaystyle\frac{\nabla f}{\|\nabla f\|^2}$ with $\gamma^k(t_1)=x^k.$ Then for any $t_2\in (t_1,b)$, the following statements hold:
\begin{enumerate}[{\rm (i)}] 
\item The trajectory $\gamma^k(t)$ reaches the fiber $f^{-1}({t_2})$ at the time $t_2$.
\item The sequence $\gamma^k(t_2)$ tends to infinity as $k\to+\infty$.
\item For any cluster point $v$ of the sequence $\displaystyle\frac{\gamma^k(t_2)}{\|\gamma^k(t_2)\|}$, we have 
$$\|u-v\|\leqslant \frac{2}{C}|t_1-t_2|.$$
\end{enumerate}
\end{lemma}

\begin{proof}
(i) For each $k,$ set
$$T_k=\sup\{t:\ t_1<t\leqslant b \text{ and } \|\gamma^k(s)\|\geqslant R \text{ for all } s\in[t_1,t]\}.$$
For all $t \in [t_1,T_k)$ we have $f(\gamma^k(t)) = t$ and 
$$\|\gamma^k(t)\|'\leqslant\|(\gamma^k)'(t)\|=\frac{1}{\|\nabla f (\gamma^k(t))\|}\leqslant\frac{\|\gamma^k(t)\|}{C}.$$ 
By applying Gr\"onwall's Lemma, it is not hard to see that (see also \cite[Theorem 3.5]{Acunto2005-1}) 
\begin{equation}\label{Pt1}
\|\gamma^k(t)\| \leqslant  \|\gamma^k(t_1)\|\exp\left(\int_{t_1}^t\frac{ds}{C}\right)=\|\gamma^k(t_1)\|e^\frac{t-t_1}{C}.
\end{equation}
On the other side, we have 
$$\gamma^k(t)-\gamma^k(t_1)=\int_{t_1}^t(\gamma^k)'(s)ds=\int_{t_1}^t\frac{\nabla f(\gamma^k(s))}{\|\nabla f(\gamma^k(s))\|^2}ds.$$
Thus
\begin{eqnarray}
\|\gamma^k(t)\| & \geqslant & \|\gamma^k(t_1)\|-\displaystyle\int_{t_1}^t\frac{ds}{\|\nabla f(\gamma^k(s))\|} \nonumber \\
& \geqslant & \|\gamma^k(t_1)\|-\displaystyle\int_{t_1}^t\frac{\|\gamma^k(s)\|}{C}ds \nonumber \\
& \geqslant & \|\gamma^k(t_1)\|-\displaystyle\int_{t_1}^t\frac{\|\gamma^k(t_1)\|}{C}e^\frac{s-t_1}{C}ds \nonumber \\
& = & \|\gamma^k(t_1)\|\left(1-\displaystyle\int_{t_1}^t d e^\frac{s-t_1}{C}\right) \nonumber \\
& = & \|\gamma^k(t_1)\| \left(1-e^\frac{s-t_1}{C}\Big|_{t_1}^{t}\right)=\|\gamma^k(t_1)\|(2-e^\frac{t-t_1}{C}) \label{tt1},
\end{eqnarray}
where the third inequality follows from~\eqref{Pt1}. Assume that $T_k < b.$ Then
$$\|\gamma^k(T_k)\|\geqslant \|\gamma^k(t_1)\|(2-e^\frac{T_k-t_1}{C})> \|\gamma^k(t_1)\|(2-e^\frac{b-a}{C})>\frac{\|\gamma^k(t_1)\|}{2} \geqslant R.$$
By continuity, it follows that $\gamma^k(T_k + \delta) \geqslant R$ for all $\delta  > 0$ small enough, which contradicts the definition of $T_k.$ Therefore  $T_k = b.$

As $\nabla f(\gamma^k(t))\ne 0$ for $t\in[t_1,T_k)=[t_1,b)\supset [t_1,t_2]$, the trajectory $\gamma^k(t)$ can not reach a stationary point before reaching the fiber $f^{-1}({t_2}).$ Moreover, since $t_1 < t_2 < b,$ it follows from \eqref{Pt1} that $\gamma^k(t)$ can not go to infinity as $t$ tends to $t_2$. Therefore, it must reach the fiber $f^{-1}({t_2})$ at the time $t_2.$

(ii) We have $\gamma^k(t_2)\in  f^{-1}({t_2})$ by item~(i). Moreover, it follows from \eqref{tt1} that
$$\|\gamma^k(t_2)\|\geqslant\|\gamma^k(t_1)\|(2-e^\frac{t_1-t_2}{C})\geqslant\|\gamma^k(t_1)\|(2-e^\frac{b-a}{C})>\frac{\|\gamma^k(t_1)\|}{2} = \frac{\|x^k\|}{2}.$$
So $\gamma^k(t_2)\to \infty$ as $k\to+\infty$. 

(iii) We know that for all $ t \in [t_1, t_2],$
\begin{equation*}
\|\gamma^k(t)\| \geqslant R \quad \text{ and } \quad f(\gamma^k(t)) = t.
\end{equation*}
Then Lemma~\ref{length}, together with~\eqref{C}, yields
$$\left\|\frac{x^k}{\|x^k\|}-\frac{\gamma^k(t_2)}{\|\gamma^k(t_2)\|}\right\|\leqslant \frac{2}{C} |t_1-t_2|.$$
Letting $k \to \infty$, we get 
$$\|u - v\| \leqslant \frac{2}{C} |t_1 - t_2|$$
for any cluster point $v$ of the sequence  $\displaystyle\frac{\gamma^k(t_2)}{\|\gamma^k(t_2)\|}.$
\end{proof}

The first main result of this paper reads as follows.

\begin{theorem}\label{LocallyLipschitzHausdorff} 
Let $f \colon \mathbb{R}^n \rightarrow \mathbb{R}$ be a polynomial function and let $t_0\not\in K_\infty(f).$ Then there exist some constants $c>0$ and $\delta>0$ such that for all $t_1,t_2\in(t_0-\delta,t_0+\delta)$, we have 
$$\dist^{g}_{H}(D_\infty({t_1}),D_\infty({t_2}))\leqslant c|t_1-t_2|,$$
where $\dist_{H}^g(\cdot, \cdot)$ denotes the Hausdorff distance with respect to the intrinsic metric in $\mathbf{D}_\infty^a.$
\end{theorem}

\begin{proof} 
Since $t_0\not\in K_\infty(f),$ there exist some constants $C>0, R > 0$ and $\delta > 0$ such that 
$$\|x\|\|\nabla f(x)\|\geqslant C$$
for all $x \in \mathbb{R}^n$ with $\|x\|\geqslant R$ and $|f(x) - t_0| < \delta.$ By shrinking $\delta$ if necessary, we can assume that $e^\frac{2\delta}{C} < \displaystyle\frac{3}{2}$ and $\displaystyle\delta < \frac{R'}{2c},$
where $\displaystyle c := \frac{2}{C}$ and 
$$R' := \min\{\dist(Z,Z'):\ Z\ne Z', \  Z \text{ and } Z' \text{ are connected components of } \mathbf{D}_\infty^a \}.$$ 
(If $\mathbf{D}_\infty^a $ is connected, we let $R' := +\infty$.)

Denote by $\dist^g(\cdot, \cdot)$ the intrinsic metric in $\mathbf{D}_\infty^a.$ Let $t_1,t_2\in(t_0-\delta,t_0+\delta)$ with $t_1 < t_2.$
To prove the theorem, it is enough to show that 
$$\dist^g(u, D_\infty({t_2}))\leqslant c|t_1-t_2| \quad \textrm{ for all } \quad u \in D_\infty({t_1}).$$
To this end, fix any $u \in D_\infty({t_1}).$ By definition, there is a sequence $x^k \in f^{-1}({t_1})$ such that $x^k\to\infty$ and $\displaystyle\frac{x^k}{\|x^k\|}\to u.$ 
For each $k,$ let $\gamma^k(\cdot)$ be the maximal integral curve of the vector field $\displaystyle\frac{\nabla f}{\|\nabla f\|^2}$ with $\gamma^k(t_1) = x^k.$ 
Passing to a subsequence if necessary, we can suppose that the sequence  $\displaystyle\frac{\gamma^k(t_2)}{\|\gamma^k(t_2)\|}$  converges to a vector $v.$ By Lemma~\ref{NotFar}, we have $v \in D_\infty({t_2})$ and 
\begin{equation}\label{u-v}
\|u - v \| \leqslant c|t_1 - t_2| \leqslant 2c \delta < R'.
\end{equation}
Consequently, $u$ and $v$ lie in a same connected component of $\mathbf{D}_\infty^a.$

Take arbitrarily $M > 1.$ By Proposition~\ref{WhitneyProperty}, there exists a finite semi-algebraic stratification $\SA$ of $\mathbf{D}_\infty^a$ such that 
\begin{equation}\label{Whitney}
\dist^g(x,y)\leqslant M\|x - y\|, 
\end{equation}
for any stratum $Y\in\SA$ and any two points $x,y\in \overline Y$. 

Assume that there exists a finite sequence of points in $\mathbf{D}_\infty^a$: 
\begin{equation}\label{SQ}
x(s_0) := u , x(s_1), \ldots, x(s_p) := v, 
\end{equation}
with $t_1 =: s_0 \leqslant s_1 < \cdots < s_p := t_2$, such that for $i=0,\ldots,p-1$, the following two properties hold:
\begin{itemize}
\item $x(s_i)$ and $x(s_{i+1})$ lie in the closure of a same stratum of $\SA,$ and
\item $\| x(s_i) - x(s_{i+1})\|  \leqslant c|s_i-s_{i+1}|.$
\end{itemize}
Then, by the inequality~\eqref{Whitney}, we get
\begin{eqnarray*}
\dist^g(x(s_i),x(s_{i+1})) &\leqslant & M\| x(s_i) - x(s_{i+1})\|  \ \leqslant \  c M |s_i-s_{i+1}|,
\end{eqnarray*}
which yields 
\begin{equation}\label{M2}
\dist^g(u,v)\leqslant\sum_{i=0}^{p-1}\dist^g(x(s_i),x(s_{i+1}))\leqslant\sum_{i=0}^{p-1}cM|s_i-s_{i+1}|=cM|t_1-t_2|. 
\end{equation}
Therefore 
$$\dist^g(u,D_\infty({t_1}))\leqslant\dist^g(u,v)\leqslant cM|t_1-t_2|.$$ 
As the inequalities hold for any $M > 1,$ they still holds for $M = 1.$  Hence it remains to construct a sequence with the required properties.

Since $u \in \mathbf{D}_\infty^a,$ there exists a stratum $Y_1 \in \SA$ such that $u \in \overline Y_1.$ If $v \in \overline Y_1,$ there is nothing prove, so assume that $v \not \in \overline Y_1.$

Let $s_0 := t_1, x(s_0) := u,$ and
$$s_1 := \sup\left\{s \in [s_0, t_2]:\ \text{ the sequence } \frac{\gamma^k(s)}{\|\gamma^k(s)\|} \text{ has a cluster point in } \overline Y_1\right\}.$$
The following claim is a key of the proof since it allows to determine the second point of the desired sequence.

\begin{claim}\label{Boundary}
There exists a cluster point of the sequence $\displaystyle\frac{\gamma^k(s_1)}{\|\gamma^k(s_1)\|}$ in $\overline Y_1.$
\end{claim}

\begin{proof} 
Observe that the statement is clear if $s_1 = s_0$ so assume that $s_1 >  s_0$. 
Assume for contradiction that the contrary holds. 
Accordingly, there is $a>0$ such that for all $k$ large enough, $\dist\left(\displaystyle\frac{\gamma^k(s_1)}{\|\gamma^k(s_1)\|},\overline Y_1\right)\geqslant a.$ Take any $N \geqslant 2$ and let $\displaystyle s_1' := s_1 - \frac{a}{c N} < s_1.$ Increasing $N$ if necessary so that $s_1' \geqslant s_0.$
In light of Lemma~\ref{NotFar}, we obtain
$$\left \|\frac{\gamma^{k}(s_1)}{\|\gamma^{k}(s_1)\|}-\frac{\gamma^{k}(s_1')}{\|\gamma^{k}(s_1')\|}\right\|\leqslant c|s_1-s_1'|=\frac{a}{N}.$$
Consequently, for all $k$ large enough, we have
$$\dist\left(\frac{\gamma^{k}(s_1')}{\|\gamma^{k}(s_1')\|},\overline Y_1\right)\geqslant\dist\left(\frac{\gamma^{k}(s_1)}{\|\gamma^{k}(s_1)\|},\overline Y_1\right) - \left\|\frac{\gamma^{k}(s_1)}{\|\gamma^{k}(s_1)\|}-\frac{\gamma^{k}(s_1')}{\|\gamma^{k}(s_1')\|}\right\|\geqslant a - \frac{a}{N} > 0.$$
It follows that $\displaystyle\frac{\gamma^{k}(s_1')}{\|\gamma^{k}(s_1')\|}$ does not have cluster points in $\overline Y_1.$
Since this fact holds for all $N$ large enough, we get a contradiction to the definition of $s_1.$ Therefore, the sequence $\displaystyle\frac{\gamma^k(s_1)}{\|\gamma^k(s_1)\|}$ must have a cluster point $x(s_1)$ in $\overline Y_1.$
\end{proof}

In light of Claim~\ref{Boundary}, we can find a cluster point $x(s_1)$ of the sequence $\displaystyle\frac{\gamma^k(s_1)}{\|\gamma^k(s_1)\|}$ in $\overline Y_1.$ 
Passing to a subsequence if necessary, we can assume that $\displaystyle\frac{\gamma^k(s_1)}{\|\gamma^k(s_1)\|}$ converges to $x(s_1)$ as $k\to+\infty.$ In order to define the next point of our sequence, we need to show that $x(s_1)$ is not a ``death end" in the boundary of $Y_1$, i.e., $Y_1$ is not the unique stratum in $\SA$ such that $x(s_1)\in\overline Y_1.$

\begin{claim}\label{Unisolated} 
There is a stratum $Y_2 \in \SA$ with $Y_2 \ne Y_1$ such that $x(s_1)\in \overline Y_2.$ 
\end{claim}

\begin{proof} 
Since $v \not \in \overline Y_1,$ we must have $s_1 < t_2.$ For each $N,$ let $\displaystyle\tau_{N} := s_1+\frac{1}{N}$ with $N$ large enough so that $\tau_N \leqslant t_2$ and let $w_N$ be a cluster point of the sequence $\displaystyle\frac{\gamma^k(\tau_N)}{\|\gamma^k(\tau_N)\|}$. It follows from Lemma~\ref{NotFar} that 
$w_N \in D_\infty({\tau_N})$ and that
$$\|w_N - x(s_1)\|\leqslant c|\tau_N - s_1|=\frac{c}{N}\to 0 \quad \textrm{ as } \quad N \to +\infty.$$ 
Hence the sequence $w_N$ converges to $x(s_1)$ as $N \to+\infty$. Since $w_N \in D_\infty({\tau_N}) \subset \mathbf{D}_\infty^a,$ there is a stratum $Y_2 \in \SA$ such that $Y_2$ contains infinite number of points of the sequence $w_N.$ Clearly, $x(s_1)\in \overline Y_2$. Note that, by definition of $s_1$, for all $N$ sufficiently large, the sequence $\displaystyle\frac{\gamma^k(\tau_N)}{\|\gamma^k(\tau_N)\|}$ does not have cluster points in $\overline Y_1$, so $w_N \not \in \overline Y_1$. Consequently $Y_2 \ne Y_1$ and the lemma is proved.
\end{proof}

If $v \in \overline Y_2,$ then it is clear that the sequence $u = x(s_0), x(s_1), v$ has the desired properties. So assume that $v \not \in \overline Y_2.$ Let
$$s_2:=\sup\{s\in[s_1,t_2] :\ \text{ the sequence } \frac{\gamma^k(s)}{\|\gamma^k(s)\|} \text{ has a cluster point in } \overline Y_2\},$$
and repeat the arguments in Claims~\ref{Boundary}~and~\ref{Unisolated} to get a point $x(s_2) \in \overline{Y}_2$ and a stratum $Y_3 \in \SA$ with $Y_3 \ne Y_2$ such that $x(s_2) \in \overline{Y}_3.$ Note that $s_1 < s_2$ because of the existence of $w_N$ in ${Y}_2$ for some $N $ large enough. (Recall that $w_N$ is a cluster point of the sequence $\displaystyle\frac{\gamma^k(\tau_N)}{\|\gamma^k(\tau_N)\|}$ with $\displaystyle\tau_N = s_1 + \frac{1}{N}.$)  Furthermore, by definition of $s_1$, the sequence $\displaystyle\frac{\gamma^k(s)}{\|\gamma^k(s)\|}$ does not have cluster points in $\overline Y_1$ for all $s>s_1$. So by induction, for each $i$, we can construct a sequence of points $u = x(s_0), x(s_1),\ldots, x(s_i)$ in $\mathbf{D}_\infty^a$ and a sequence of strata $Y_1, \ldots, Y_{i + 1}$ in $\SA$ satisfying the following conditions:
\begin{itemize}
\item $x(s_i) = \lim_{k \to \infty} \displaystyle\frac{\gamma^k(s_i)}{\|\gamma^k(s_i)\|} \in \overline{Y}_{i},$ 
\item $x(s_i)$ and $x(s_{i+1})$ lie in $\overline{Y}_{i + 1}$ (assuming that $v \not \in \overline{Y}_{i}),$ and
\item $\displaystyle\frac{\gamma^k(s)}{\|\gamma^k(s)\|}$ does not have cluster points in $\bigcup_{j=1}^i\overline Y_j$ for all $s > s_i.$
\end{itemize}
The first condition and Lemma~\ref{NotFar} together imply that
$$\| x(s_i) - x(s_{i+1})\|  \leqslant c|s_i-s_{i+1}|.$$
The third condition shows that the strata $Y_j$ are distinct. Since there is only a finite number of strata in $\SA,$ there must exist $p > 0$ such that $x(s_{p-1})$ and $v$ lie in $\overline Y_p$, whence the sequence 
$$u, x(s_1), \ldots, x(s_{p - 1}), v$$ 
has the desired properties. This ends the proof of the theorem.
\end{proof}

\begin{remark}
From the proof of Theorem~\ref{LocallyLipschitzHausdorff} we also have that $x(s_i) \in \overline{Y}_i \setminus Y_i$ for $i = 1, \ldots, p - 1.$ Since we do not use this fact, we leave the proof to the reader.
\end{remark}

As a consequence of Theorem~\ref{LocallyLipschitzHausdorff}, we can see that the set of tangent directions at infinity of the fiber of a polynomial function varies (locally Lipschitz) continuously except at a finite number of values.

\begin{corollary}\label{LocallyLipschitz} 
Let $f \colon \mathbb{R}^n \rightarrow \mathbb{R}$ be a polynomial function. Then the multivalued function 
$$D_\infty \colon \mathbb{R} \rightrightarrows \mathbb{S}^{n - 1}, \quad t \mapsto D_\infty(t),$$ 
is locally Lipschitz outside $K_\infty(f)$, i.e., for each $t_0\not\in K_\infty(f)$, there exist some constants $c>0$ and $\delta>0$ such that for all $t_1,t_2\in(t_0-\delta,t_0+\delta)$, we have 
\begin{equation}\label{LocallyLipschitzE}
D_\infty ({t_1})\subset D_\infty ({t_2})+c|t_1-t_2|\B^n.
\end{equation}
In addition, the mapping $t\mapsto\dim D_\infty ({t})$ is lower semicontinuous at $t_0.$
\end{corollary}

\begin{proof} 
The first assertion follows directly from Theorem~\ref{LocallyLipschitzHausdorff}. It remains to show that 
$$\dim D_\infty ({t})\geqslant\dim D_\infty ({t_0})$$ 
for $t$ close enough to $t_0.$ If $D_\infty ({t_0})=\emptyset$ then there is nothing to prove; so assume that $D_\infty ({t_0}) \ne \emptyset.$ 
Clearly it suffices to consider only the case where $t>t_0$. Let 
$$A_{(t_0,t_0+\delta)}:=\{(u, t)\in \mathbb S^{n-1}\times (t_0,t_0+\delta):\ u \in D_\infty(t)\},$$
which is a semi-algebraic set. Let $\pi \colon A_{(t_0,t_0+\delta)} \to (t_0,t_0+\delta)$ be the projection on the last coordinate. Obviously 
$$\pi^{-1}(t)=A_{(t_0,t_0+\delta)}\cap (\mathbb S^{n-1}\times\{t\})=D_\infty(t)\times\{t\}.$$ 
In light of Theorem~\ref{HardtTheorem}, there exists a positive constant $\delta'\leqslant\delta$ such that $\pi$ is a semi-algebraic fibration on $A_{(t_0,t_0+\delta')}.$ Consequently, the function $t\mapsto\dim D_\infty ({t})$ is constant on $(t_0,t_0+\delta').$ For $t\in (t_0,t_0+\delta')$, observe that $\dim D_\infty ({t}) \ne -1$ since otherwise, $f^{-1}({t})$ is compact while $f^{-1}({t_0})$ is not which implies that $t_0$ is a bifurcation value of $f$, so $t_0\in K_\infty(f)$ which is a contradiction. Therefore $D_\infty ({t}) \ne \emptyset$, which yields 
$$\dim A_{(t_0,t_0+\delta')}=\dim D_\infty ({t}) + 1.$$
Now by \eqref{LocallyLipschitzE}, it is not hard to see that $D_\infty({t_0})\times\{t_0\}\subset\partial A_{(t_0,t_0+\delta')}.$ Thus, by~\cite[Proposition~3.16]{Coste2000-1},
$$\dim D_\infty({t_0})\leqslant \dim A_{(t_0,t_0+\delta')}-1=\dim D_\infty ({t}).$$ 
This finishes the proof of the corollary.
\end{proof}

As a consequence of Theorem~\ref{LocallyLipschitzHausdorff}, we deduce below that the volume function 
$$\mathbb{R} \to \mathbb{R}, \quad t\mapsto \vol_{n-2}(D_\infty(t))$$ 
is locally Lipschitz outside the set $K_\infty(f)$. The main idea of the proof is that the entropy of any  semi-algebraic set having ``small width" can be estimated by the entropy of a semi-algebraic set of lower dimension. Since the dimension of $D_\infty(t)$ is at most $n-2$ by Lemma~\ref{GeometricAlgebraic}, it is natural to consider the volume in this dimension.

\begin{theorem}\label{VolumeContinuous} 
Let $f \colon \mathbb{R}^n \rightarrow \mathbb{R}$ be a polynomial function. Then the volume function 
$$\mathbb{R} \rightarrow \mathbb{R}, \quad t\mapsto \vol_{n-2}(D_\infty(t)),$$ 
is locally Lipschitz outside the set $K_\infty(f).$
\end{theorem}
\begin{proof} 
Fix $t_0 \not\in K_\infty(f)$ and let $c>0$ and $\delta>0$ be the constants determined in Theorem~\ref{LocallyLipschitzHausdorff}.  Since Theorem~\ref{LocallyLipschitzHausdorff} still holds if we shrink $\delta$ and increase $c$, we can suppose that $\delta\leqslant\frac{1}{2}$ and $c\geqslant 1.$ 
For $t_1,t_2\in (t_0-\delta,t_0+\delta)$, one has 
\begin{eqnarray*}
|\vol_{n-2}(D_\infty({t_2}))  \ - && \hspace{-1cm}\ \vol_{n-2}(D_\infty({t_1}))|\\
&=&|\vol_{n-2}(D_\infty({t_2})\setminus D_\infty({t_1}))+\vol_{n-2}(D_\infty({t_2})\cap D_\infty({t_1}))-\\
&&(\vol_{n-2}(D_\infty({t_1})\setminus D_\infty({t_2}))+\vol_{n-2}(D_\infty({t_1})\cap D_\infty({t_2})))|\\
&=&|\vol_{n-2}(D_\infty({t_2})\setminus D_\infty({t_1}))-\vol_{n-2}((D_\infty({t_1})\setminus D_\infty({t_2}))|\\
&\leqslant&\vol_{n-2}((D_\infty({t_1})\setminus D_\infty({t_2}))+\vol_{n-2}(D_\infty({t_2})\setminus D_\infty({t_1})).
\end{eqnarray*}
Now the proof is completed by demonstrating the following inequalities:
\begin{eqnarray}
\label{1lim}\vol_{n-2}((D_\infty({t_2})\setminus D_\infty({t_1})) &\leqslant& a|t_1-t_2|, \\
\label{2lim}\vol_{n-2}((D_\infty({t_1})\setminus D_\infty({t_2})) &\leqslant& a|t_1-t_2|,
\end{eqnarray}
where $a$ is a positive constant not depending on $t.$

We will prove only~\eqref{1lim} since proving~\eqref{2lim} is completely similar. Observe that~\eqref{1lim} is trivial if $\dim ((D_\infty({t_2})\setminus D_\infty({t_1}))<n-2$, so suppose that $\dim ((D_\infty({t_2})\setminus D_\infty({t_1}))=n-2$. 
For each $t\in\mathbb R$, let $X_{t}\subseteq D_\infty({t})$ be the semi-algebraic set defined in Lemma~\ref{GeometricAlgebraic}.
Then $D_\infty({t}) \setminus X_{t}$ is a semi-algebraic set. 
By Lemma~\ref{GeometricAlgebraic}, there are representations of $D_\infty({t})$ and $X_t$ such that the diagrams $\mathscr{D}(D_\infty({t}))$ and $\mathscr{D}(X_{t})$ depend only on $n$ and the degree of $f$. Hence $\mathscr{D}(D_\infty({t})\setminus X_t)$ also depends only on $n$ and the degree of $f$.
Consequently, we can fix a diagram
$$\mathscr{D} := \mathscr{D}(D_\infty({t}))\setminus X_{t}$$ 
depending only on $n$ and the degree of $f$. 

We claim that 
\begin{equation}\label{d=}
\dist^{g}(u,D_\infty({t_1})\setminus X_{t_1})=\dist^{g}(u,D_\infty({t_1}))
\end{equation}
for any $u\in D_\infty({t_2})\setminus D_\infty({t_1}).$ To this end, take arbitrarily a continuous curve $\gamma\colon[0,1]\to \mathbf{D}_\infty^a$ such that $\gamma(0)=u$ and $\gamma(1)\in D_\infty({t_1}).$ There exists $T \in (0,1]$ such that $\gamma(T)\in D_\infty({t_1})$ and $\gamma(t)\not\in D_\infty({t_1})$ for $t\in[0,T).$ 
Then clearly $\gamma(T)\in D_\infty({t_1})\setminus X_{t_1}.$ 
Hence
$$\dist^{g}(u,D_\infty({t_1})\setminus X_{t_1})\leqslant\dist^{g}(u,D_\infty({t_1})).$$
Observe that the inequality
$$\dist^{g}(u,D_\infty({t_1})\setminus X_{t_1})\geqslant\dist^{g}(u,D_\infty({t_1}))$$
is obvious so the inequality~\eqref{d=} holds.

By Theorem~\ref{LocallyLipschitzHausdorff} and~\eqref{d=}, we have
$$\dist(u,D_\infty({t_1})\setminus X_{t_1})\leqslant\dist^{g}(u,D_\infty({t_1})\setminus X_{t_1})=\dist^{g}(u,D_\infty({t_1}))\leqslant c|t_1-t_2|,$$
i.e., $u \in \mathscr{N}_{c|t_1-t_2|}(D_\infty({t_1})\setminus X_{t_1}).$ Consequently 
$$D_\infty({t_2})\setminus D_\infty({t_1})\subset \mathscr{N}_{c|t_1-t_2|}(D_\infty({t_1})\setminus X_{t_1}).$$ 
Set $l:=\dim(D_\infty({t_1})\setminus X_{t_1})$. Clearly $l<n-2$. 
Applying Proposition~\ref{5.8}, we get the following upper bound for the $c|t_1-t_2|$-entropy of $D_\infty({t_2})\setminus D_\infty({t_1})$:
\begin{equation}\label{4.1}\begin{array}{lll}
M(c|t_1-t_2|,D_\infty({t_2})\setminus D_\infty({t_1})) 
&\leqslant& \displaystyle\left(\frac{4}{c|t_1-t_2|}\right)^{l}C(n)\nu(l)\alpha(n,\mathscr{D}(D_\infty({t_1})\setminus X_{t_1}))\\
&=&         \displaystyle\left(\frac{4}{c|t_1-t_2|}\right)^{l}C(n)\nu(l)\alpha(n,\mathscr{D}),
\end{array}\end{equation}
where $C(n)$ and $\nu(l)$ are the positive constants defined in Proposition~\ref{5.8}, which depend only on $n$ and $l$ respectively; $\alpha(n,\mathscr{D})$ is a positive constant depending only on $n$ and $\mathscr{D}.$

On the other side, by Proposition~\ref{5.14}, there is a positive constant $C_1$ depending only on the diagram of $D_\infty({t_2})\setminus D_\infty({t_1})$ such that 
\begin{equation}\label{4.2}C_1\sum_{i=1}^{n-2}V_i(D_\infty({t_2})\setminus D_\infty({t_1}))\left(\frac{1}{c|t_1-t_2|}\right)^{i}\leqslant M(c|t_1-t_2|,D_\infty({t_2})\setminus D_\infty({t_1})),
\end{equation}
where $V_i(D_\infty({t_2})\setminus D_\infty({t_1}))$ is the $i$-th variation of $D_\infty({t_2})\setminus D_\infty({t_1})$. Note that the diagram of $D_\infty({t})$ does not depend on $t$ in light of~Lemma~\ref{GeometricAlgebraic}, so neither does $C_1$. Now, combining~\eqref{4.1}~and~\eqref{4.2}, we have
\begin{eqnarray*}
\vol_{n-2}(D_\infty({t_2})\setminus D_\infty({t_1}))\left(\frac{1}{c|t_1-t_2|}\right)^{n-2} 
&\leqslant & \sum_{i=1}^{n-2}V_i(D_\infty({t_2})\setminus D_\infty({t_1}))\left(\frac{1}{c|t_1-t_2|}\right)^{i}\\
&\leqslant & \frac{1}{C_1}M(c|t_1-t_2|,D_\infty({t_2})\setminus D_\infty({t_1}))\\
&\leqslant&\frac{1}{C_1}\left(\frac{4}{c|t_1-t_2|}\right)^{l}C(n)\nu(l)\alpha(n, \mathscr{D}).
\end{eqnarray*}
Therefore 
\begin{eqnarray*}
\vol_{n-2}(D_\infty({t_2})\setminus D_\infty({t_1}))&\leqslant & \frac{4^l}{C_1}(c|t_1-t_2|)^{n-2-l}C(n)\nu(l)\alpha(n, \mathscr{D})\\
&= &\frac{4^l}{C_1}c^{n-2-l}(|t_1-t_2|)^{n-3-l}C(n)\nu(l)\alpha(n, \mathscr{D})|t_1-t_2|\\
&\leqslant &\frac{4^{n - 2 - 1}}{C_1}c^{n-2}C(n)\nu(n-3)\alpha(n, \mathscr{D})|t_1-t_2|,
\end{eqnarray*}
where the last inequality follows from the following facts: 
$$|t_1-t_2|<2\delta\leqslant 1,\ c\geqslant 1,\ n-3\geqslant l \text{ and } \nu(n-3)\geqslant\nu(l).$$ 
This completes the proof of \eqref{1lim} and hence that of the theorem. 
\end{proof}


\noindent{\bf Questions.} 
\begin{enumerate}[{\rm (1)}]
\item Is the function $t\mapsto\dim D_\infty ({t})$ constant on each connected component of $\mathbb R\setminus K_\infty(f)?$
\item Does $D_\infty(t)$ have the same topology for all $t$ in a connected component of $K_\infty(f)$?
\item How about the cases of polynomial mappings, semi-algebraic or definable functions and mappings?
\end{enumerate}

\subsection*{Acknowledgments} 
We would like to thank Olivier Le Gal for some helpful discussion during the preparation the paper. This work was partially performed while the author$\dagger$ had been visiting the laboratory LAMA $-$ Universit\'e Savoie Mont Blanc $-$ CNRS research unit number 5127 by benefiting a ``poste rouge" of the CNRS. The author$\dagger$ would like to thank the laboratory, INSMI and LIA Formath Vietnam (CNRS) for hospitality and support.

\bibliographystyle{abbrv}


\end{document}